\documentclass[10pt]{article}
\usepackage{geometry,fullpage}
\usepackage{amssymb,amsmath, hyperref, amsthm, amsfonts,graphicx,mathabx,cleveref,lipsum}
\usepackage{authblk}
\usepackage{enumitem,xcolor}
\usepackage{dirtytalk}
\title{Sherman-Takeda type theorems for locally $C^*$-algebras.}
\author[1]{Lav Kumar Singh}
\author[1,2]{ Aljo\v{s}a Peperko}

\affil[1]{Institute of Mathematics, Physics and Mechanics, Jadranska ulica 19, SI-1000, Ljubljana, Slovenia.}
\affil[2]{Faculty of Mechanical Engineering, University of Ljubljana, A\v{s}ker\v{c}eva 6, SI-1000, Ljubljana, Slovenia.}
\date{}
\newtheorem{theorem}{Theorem}[section]
\newtheorem{corollary}{Corollary}[theorem]
\newtheorem{lemma}[theorem]{Lemma}
\newcommand{\mA}{\mathcal A}
\theoremstyle{definition}
\newtheorem{example}{Example}[section]
\newtheorem{remark}{Remark}[section]
\newtheorem{definition}{Definition}[section]
\DeclareRobustCommand{\looongrightarrow}{%
	\DOTSB\relbar\joinrel\relbar\joinrel\relbar\joinrel\rightarrow}
\begin{document}
	\newcommand{\Addresses}{{
			\bigskip
			\footnotesize
			
			Lav Kumar Singh, \textsc{Institute of Mathematics, Physics and Mechanics, Jadranska ulica 19, SI-1000, Ljubljana, Slovenia.}\par\nopagebreak
			\textit{E-mail address}: \texttt{lavksingh@hotmail.com}
			
			\medskip
			
			Aljo\v{s}a Peperko, \textsc{Institute of Mathematics, Physics and Mechanics, Jadranska Ulica, Ljubljana, Slovenia. Faculty of Mechanical Engineering, A\v{s}ker\v{c}eva 6, SI-1000, Ljubljana, Slovenia}\par\nopagebreak
			\textit{E-mail address}: \texttt{aljosa.peperko@fs.uni-lj.si}		
	}}
	
	\maketitle
	\begin{abstract}
		In this article, we will first establish some density results for a locally $C^*$-algebra $\mathcal A$ and then identify a property, called Kaplansky density property (KDP). We then give a induced faithful continuous $*$-representation $\varphi$ of $\mathcal A^{**}$ (equipped with unique Arens product) on the space $B_{loc}(\mathcal H)$ such that $\varphi(\mathcal A^{**})\subset \overline{\pi(\mathcal A)}^{WOT}$, where $\pi:\mathcal A\to B_{loc}(\mathcal H)$ is the associated universal $*$-representation and $\mathcal H$ is the associated locally Hilbert space. Finally we show that for a Fr\'echet locally $C^*$-algebra $\mathcal A$ possessing KDP, the second strong dual is algebraically and topologically $*$-isomorphic to $ \overline{\pi(\mathcal A)}^{WOT}$, which is a direct analogue of the classical Sherman-Takeda theorem for $C^*$-algebras. We shall also observe the joint continuity of some associated bilinear maps in the running.
	\end{abstract}
	
	\section{Introduction}
	The second dual space of $C^*$-algebras and operator algebras has been extensively studied by several mathematicians over the years. A cornerstone result regarding the second dual of a $C^*$-algebra is the Sherman-Takeda theorem, first noted in \cite{Sherman} and \cite{Takeda}, and later by many others through different techniques. The theorem states that the second dual space $\mathcal A^{**}$ of any $C^*$-algebra $\mathcal A$ can be identified as an enveloping von Neumann algebra of $\mathcal A$. The following theorem precisely captures the essence of it
	\begin{theorem}[Sherman-Takeda]\label{STT}
		If $\mathcal A$ is a $C^*$-algebra and $\pi:\mathcal A\to B(\mathcal H)$ is its universal representation on a Hilbert space $\mathcal H$, then the second dual $\mathcal A^{**}$ is a von Neumann algebra which is isometrically $*$-isomorphic to $\pi(\mathcal A)^{cc}=\overline{\pi(\mathcal A)}^{WOT}$, where $\mathcal \pi(\mathcal A)^{cc}$ denotes the double commutant of $\pi(\mathcal A)$. 
	\end{theorem}
	\noindent	Further, it is known that the multiplication on $\mathcal A^{**}$ acquired from $\pi(\mathcal A)^{cc}$ agrees with either Arens products on $\mathcal A^{**}$ (both the Arens products are same on the second dual of a $C^*$-algebra). \\
	
	In this article, we study the bidual of locally $C^*$-algebras. The notion of \say{locally $C^*$-algebra} was given by A. Inoue in \cite{Inoue}. These objects are complete locally multiplicatively convex (lmc) topological algebras and are the inverse limits of $C^*$-algebras. They retain many properties of $C^*$-algebras despite not having a norm which generates its topology. Locally $C^*$-algebras also go by the name $pro~C^*\text{-}algebras$ in various literature like $\cite{Maria, Philip, Bhatt}$. Over the past few years, locally $C^*$-algebras have garnered significant interest among functional analysts (see \cite{Gheondea}). A notion of locally $W^*$-algebras was developed by M. Fragoulopulo\cite{Maria} and M. Joita\cite{Joita} parallel to the theory of von-Neumann algebras. The following natural questions persist in the view of Sherman-Takeda theorem for $C^*$-algebras.
	\begin{enumerate}[noitemsep] \item What is the structure of the second dual space of a locally $C^*$-algebra? \item Whether an analogous Sherman-Takeda type theorem holds true for locally $C^*$-algebras?
	\end{enumerate}
	To tackle these problems, one should first establish an appropriate topology on the dual space of a topological vector space. The natural choice of topology on the continuous dual space of a topological vector space is the strong-dual topology $\tau_b$ (topology of convergence on  bounded sets). The continuous dual space equipped with this topology is a locally convex space. Next we must ensure the existence of Arens products on the second continuous dual space of a locally $C^*$-algebras. S. Gulick showed in \cite{Gulick} that hypocontinuity of multiplication in a topological algebra guarantees the existence of two Arens product on the second continuous dual space. Recently M. Filali, in \cite{Filali} gave even more relaxed conditions which ensures the existence of Arens product on second continuous dual of a topological algebras. All lmc-algebras (and in particular locally $C^*$-algebras) have hypocontinuous (in fact jointly continuous) multiplication and hence their second dual space admits two Arens products, with respect to which the second dual space becomes  a locally convex topological algebra. Within this setting, we shall obtain the following main results in section 4. 
	\begin{theorem}
		If a Fr\'echet locally $C^*$-algebra $\mathcal A$ has Kaplanski density property then $\mathcal A^{**}$ is topologically $*$- isomorphic to the algebra $\overline{\pi(\mathcal A)}^{WOT}$, where $\pi:\mathcal A\to B_{loc}(\mathcal H)$ is the universal representation. ( Th.\ref{main-new}).
	\end{theorem}
	
	To establish this result, we need lay down the preliminaries and some Kaplanski density type results for locally $C^*$-algebras, which we shall do in section 2 and section 3.
	\section{Preliminaries}
	Throughout this article, a topological vector space is always assumed to be Hausdorff, unless stated otherwise. By a \emph{topological algebra}, we mean a topological vector space equipped with an associative multiplication structure which makes it an algebra. We need the following notion of inverse and direct limits.
	\begin{definition}[Inverse/Projective limit]
		An \emph{inverse system} of locally convex topological vector spaces consists of a family of
		 locally convex topological vector spaces $\{\mathcal{V}_\alpha\}_{\alpha\in \mathcal I}$ indexed on a directed set $\mathcal I$, along with 
			 continuous linear maps $f_{\alpha\beta}: \mathcal{V}_\beta \to \mathcal{V}_\alpha$ for all $\alpha \leq \beta \in \mathcal I$, satisfying:
			$f_{\alpha\alpha} = \mathrm{id}_{\mathcal{V}_\alpha}$ (identity map) and
			$f_{\alpha\gamma} = f_{\alpha\beta} \circ f_{\beta\gamma}$ for $\alpha \leq \beta \leq \gamma$ (compatibility).		
		The \emph{inverse limit} $\varprojlim \mathcal{V}_\alpha$, is the subspace of the direct product $\prod_{\alpha \in I} \mathcal{V}_\alpha$:
		\[
		\varprojlim_{\alpha \in I} \mathcal{V}_\alpha = \left\{ (v_\alpha)_{\alpha \in I} \in \prod_{\alpha \in I} \mathcal{V}_\alpha \ \middle|\ f_{\alpha\beta}(v_\beta) = v_\alpha \ \forall \alpha \leq \beta  \right\}.
		\]
		The inverse limit comes equipped with canonical continuous projection maps $\pi_\alpha: \varprojlim \mathcal{V}_\alpha \to \mathcal{V}_\alpha$, defined by $\pi_\alpha((v_\beta)_{\beta}) = v_\alpha$. These satisfy $\pi_\alpha = f_{\alpha\beta} \circ \pi_\beta$ for $\alpha \leq \beta$. The topology on $\varprojlim \mathcal V_\alpha$ is actually the initial topology with respect to the projection maps which makes $\varprojlim \mathcal V_\alpha$ a locally convex topological vector space.
		The universal property of inverse limits states that
		 for any vector space $\mathcal W$ with continuous linear maps $g_\alpha: \mathcal W \to \mathcal{V}_\alpha$ compatible under the transitions (i.e., $f_{\alpha\beta} \circ g_\beta = g_\alpha$ for $\alpha \leq \beta$), there exists a unique continuous linear map $g: \mathcal W \to \varprojlim \mathcal{V}_\alpha$ such that $\pi_\alpha \circ g = g_\alpha$ for all $\alpha$.
			\end{definition}
	\begin{definition}[Direct/Inductive limit]
		Let $\mathcal I$ be a directed poset. A \emph{direct system} consists of:
		a family of locally convex topological vector spaces $\{\mathcal{V}_\alpha\}_{\alpha\in \mathcal I}$, linear continuous maps $f_{\alpha\beta}: \mathcal{V}_\alpha \to \mathcal{V}_\beta$ for all $\alpha \leq \beta \in \mathcal I$, satisfying: $f_{\alpha\alpha} = \mathrm{id}_{\mathcal{V}_\alpha}$ (identity map), $f_{\beta\gamma} \circ f_{\alpha\beta} = f_{\alpha\gamma}$ for $\alpha \leq \beta \leq \gamma$ (compatibility). 		
		The \emph{direct limit} $\varinjlim \mathcal{V}_\alpha$ is the quotient of the direct sum $\bigoplus_{\alpha \in I} \mathcal{V}_\alpha$ by the subspace $N$ generated by elements of the form $v_\alpha - f_{\alpha\beta}(v_\alpha)$ for all $\alpha \leq \beta \in \mathcal I$ and $v_\alpha \in \mathcal{V}_\alpha$:
		\[
		\varinjlim_{\alpha \in I} \mathcal{V}_\alpha = \left( \bigoplus_{\alpha \in I} \mathcal{V}_\alpha \right) \Big/ N,
		\]
		where $N = \langle v_\alpha - f_{\alpha\beta}(v_\alpha) \mid \alpha \leq \beta, \, v_\alpha \in \mathcal{V}_\alpha \rangle$.
		Each equivalence class $[ (v_\alpha) ]$ (with finitely many nonzero components) represents a compatible family where components ``induce'' consistently via the transition maps; elements are glued together across indices. Inherited from the direct sum quotient: addition and scalar multiplication are well-defined on cosets. Equivalently, it can be viewed as the disjoint union $\coprod_{\alpha \in I} \mathcal{V}_\alpha$ modulo the equivalence relation $\sim$ where $v_\alpha \sim f_{\alpha\beta}(v_\alpha)$ for $\alpha \leq \beta$. The direct limit comes equipped with canonical inclusion maps $\iota_\alpha: \mathcal{V}_\alpha \to \varinjlim \mathcal{V}_\alpha$, defined by $\iota_\alpha(v_\alpha) = [v_\alpha]$ (the class with $v_\alpha$ in the $\alpha$-component). These satisfy $\iota_\beta \circ f_{\alpha\beta} = \iota_\alpha$ for $\alpha \leq \beta$. The topology on $\varinjlim\mathcal V_\alpha$ is the finest locally convex topology with respect to these inclusions. A sub-base for inductive limit topology is given by
		$$\mathcal B=\{U\subset \varinjlim\mathcal V_\alpha, \text{convex, balanced and absorbing}~:~\iota_\alpha^{-1}(U)~\text{is neighbourhood of~} \mathcal V_\alpha~\forall \alpha\in \mathcal I\}.$$ The universal property for direct limit states that if $\mathcal W$ is another locally convex topological vector space such that for each $\alpha\in \mathcal I$ there is a continuous linear map $\eta_\alpha:\mathcal V_\alpha\to \mathcal W$, then there exists a continuous linear map $\eta:\varinjlim \mathcal V_\alpha\to \mathcal W$. The projective system $\{\mathcal V_\alpha, f_{\alpha,\beta}\}$ is said to be \emph{reduced} if $f_{\alpha,\beta}(\mathcal V_\beta)$ is a dense subspace of $\mathcal V_\alpha$ for each $\alpha<\beta$.
	\end{definition}
	\begin{lemma}\label{dualdir}
		Let $\{\mathcal V_\alpha, \iota_{\alpha,\beta}\}$ be an inductive system of locally convex spaces and $\mathcal V=\varinjlim \mathcal V_\alpha$, then we have the following isomorphism of topological vector spaces \cite[IV.4.5]{Schaefer}. $$\bigl(\mathcal V^*,\sigma(\mathcal V^*,\mathcal V)\bigr)=\varprojlim \bigl(\mathcal V_\alpha^*,\sigma(\mathcal V_\alpha^*,\mathcal V_\alpha)\bigr).$$
	\end{lemma}
	\begin{lemma} Let $\{\mathcal V_\alpha, f_{\alpha,\beta}\}$ be a reduced projective system of locally convex spaces and $\mathcal V=\varprojlim \mathcal V_\alpha$, then we have the following isomorphism of topological vector spaces $$ \bigl(\mathcal V^*,\tau (\mathcal V^*, \mathcal V)\bigr)=\varinjlim \bigl(\mathcal V_\alpha^*, \tau(\mathcal V_\alpha^*, \mathcal V_\alpha)\bigr)$$ 
		where $\tau(X,X')$ denotes the Mackey topology with respect to the dual pairing $(X,X')$ of topological vector spaces. {\cite[IV.4.4]{Schaefer}}
		\end{lemma}
			\begin{remark}
			In general,  $\bigl(\varprojlim \mathcal V_\alpha\bigr)^{**}\neq \bigl(\varprojlim \mathcal V_\alpha^{**}\bigr)$ with respect to the strong-dual topologies on dual spaces.
		\end{remark}
		\begin{definition}
		Let $\Lambda$ be an ordered poset, $\{\mathcal H_\alpha\}_{\alpha\in \Lambda}$ be a family of Hilbert spaces such that $\mathcal H_\alpha\subseteq H_\beta$ whenever $\alpha\leq \beta$ and there exists an embedding of Hilbert spaces $T_{\alpha,\beta}:\mathcal H_\alpha\to \mathcal H_\beta$ whenever $\alpha\leq \beta$. Then the \emph{locally Hilbert space} $\mathcal H$ associated to this directed system of Hilbert spaces is the direct limit $$\mathcal H=\underset{\longrightarrow}{\lim}\mathcal H_\alpha=\coprod_{\alpha\in \Lambda} \mathcal H_\alpha$$
		equipped with the inductive limit topology (finest locally convex initial topology).
	\end{definition}
	\begin{lemma}\label{limdual}
		The continuous dual $\mathcal H^*$ of a locally Hilbert space $\mathcal H=\coprod_{\alpha\in \Lambda} H_\alpha$ is isomorphic (as a vector space) to the inverse limit $\underset{\longleftarrow}{\lim}\mathcal H_\alpha^*$.
	\end{lemma}
	\begin{proof}
		The statement follows from the duality in Lemma \ref{dualdir}. The map $\varprojlim \mathcal H_\alpha^*\longrightarrow(\varinjlim \mathcal H_\alpha)^*$ defined as $$\{\left<\cdot,h_\alpha\right>\}_{\alpha\Lambda}\bigl( k_\beta\bigr)=\left<h_\beta,k_\beta\right>$$ is a bijection, where $\{\left<\cdot,h_\alpha\right>\}_{\alpha\Lambda}\in \varprojlim\mathcal H_\alpha^*$ and $k_\beta\in \varinjlim\mathcal H_\alpha$.
	\end{proof}
	The strong dual topology on $\mathcal H^*$ is distinct from the inverse limit topology. In this article the default topology on dual spaces is the strong dual topology. Since $\mathcal H$ is barrelled, the map $J:\mathcal H\to \mathcal H^{**}$ is a topological linear isomorphism on its image. If $\mathcal H$ is a locally Hilbert space which is direct limit of a countable system of Hilbert spaces, then its strong dual $\mathcal H^*$  is a Fr\'echet space, because $\mathcal H$ being a countable direct limit of Hilbert space is a DF-space \cite[Th. 12.4.8]{Jarchow} and the strong dual of a DF-space is a Fr\'echet space. Further the countable direct limit locally Hilbert space $\mathcal H$ is reflexive, i.e., $J(\mathcal H)=\mathcal H^{**}$ \cite[Prop. 11.4.5(e)]{Jarchow}. \\
	A locally Hilbert space $\mathcal H$ is also equipped with a natural inner product structure $\left<h_\alpha, h_\beta\right>=\left<h_\alpha,h_\beta\right>_{\mathcal H_\gamma}$, where $\gamma\geq \alpha, \beta$. The norm topology on $\mathcal H$ with respect to this inner product is weaker than the direct/inductive limit topology.
	\subsection{Locally multiplicatively convex algebras}
	
	\begin{definition}
		An algebra $\mathcal A$ is called locally m-convex (lmc) algebra if there exists a family $\{P_j\}_{j\in J}$ of seminorms defined on $\mathcal A$ such that \begin{enumerate}
			\item $\{P_j\}$ defines Hausdorff locally convex topology on $\mA$.
			\item $P_j(xy)\leq P_j(x)P_j(y)$ for each $x,y\in \mA$ and each $j\in J$.
		\end{enumerate}		
	
		A $*$-algebra $\mA$ is called locally m-convex (lmc) $*$-algebra if it is a lmc-algebra and $P_j(x^*)=P_j(x)$ for each $x\in \mA$ and each $j\in J$.	A complete lmc $*$-algebra $\mA$ is called a \say{locally $C^*$-algebra} if 
		\begin{enumerate}[resume]
			\item $P_j(x^*x)=P_j(x)^2$ for each $x\in\mA$ and each $j\in J$.
		\end{enumerate}
	\end{definition}
	We have the following nice characterization of locally $C^*$-algebras due to A. Inoue.
	\begin{theorem}\label{lcd}
		Every locally $C^*$-algebra $\mathcal A$ is topologically $*$-isomorphic to the inverse limit of an reduced inverse system of $C^*$-algebras. \cite{Inoue}
	\end{theorem}
	In fact if $\mathcal A$ is a locally $C^*$-algebra with associated seminorms $\{P_j\}_{j\in J}$, then for each $j$ the closure of the space ${\mathcal A_j}=\frac{\mathcal A}{N_j}$ is a $C^*$-algebra, where $N_j=\ker P_j$ for each $j$. Further, for $i<j$ there is a natural $*$-homomorphism $\pi_{ij}:\bar{A_j}\to \bar{A_i}$ such that $\pi_{ik}\circ\pi_{kj}=\pi_{ij}$ for $i\leq k\leq j$. Then $\mathcal A=\underset{\longleftarrow}{\lim }\bar{\mathcal A_j}$ (inverse limit is equipped with the initial topology and pointwise multiplication).
	
	Throughout this article we shall equip the the continuous dual $\mathcal V^{*}$ and bidual $\mathcal V^{**}$ of a topological vector space $\mathcal V$ with the strong dual topology $\tau (\mathcal V^*,\mathcal V)$ and $\tau(\mathcal V^{**},\mathcal V^*)$ respectively, (i.e., the topology of convergence on fundamental system of bounded sets of $\mathcal V$ and $\mathcal V^*$ respectively and usually also denoted by $\tau_b$). A set $X\subset \mathcal V$ is said to be \emph{(von Neumann)} bounded if for each $0$-neighbourhood $V$ in $\mathcal V$, there exists a $r_0>0$ such that $X\subset rV$ for all $r\geq r_0$. In case when $\mathcal V$ is locally convex, a set $X$ is bounded iff there exists $M_p>0$ for each continuous seminorm $p$ on $\mathcal V$ such that $p(X)\subset [0,M_p]$.
	\begin{lemma}
		Let $\mathcal V$ be a topological vector spaces and $J:\mathcal V\to \mathcal V^{**}$ denote the evaluation map. Then $J$ is injective. If $\mathcal V$ is quasibarrelled, then $J$ is an isomorphic embedding. $J(A)$ is $\sigma(\mathcal V^{**},\mathcal V^*)$ dense in $\mathcal V^{**}$. In-fact for every $v^{**}\in \mathcal V^{**}$, there exists a bounded net (all elements lie in a bounded set) $\{v_\alpha\}$ which converges to $v^{**}$ in the $\sigma(\mathcal V^{**},\mathcal V^*)$ topology. \cite[23.4.4]{Kothe}
	\end{lemma} 
	
	\begin{definition}
		A \emph{Fr\'echet space} is a complete Hausdorff locally convex topological vector space whose topology is generated by a countable family of seminorms.
	\end{definition}
	Fr\'echet spaces are metrizable. In fact, if $\mathcal V$ is a a Fr\'echet space with associated countable family of seminorms $\{P_i\}_{i\in \mathbb N}$, then the metric $$d(v_1,v_2)=\sum_{i=1}^\infty\frac{1}{2^i} \frac{P_i(v_1-v_2)}{1+P_i(v_1-v_2)}$$ generates the topology of $\mathcal V$.	We have the following well known implications for locally convex topological vector spaces (see \cite{Treves, Jarchow} for proofs).
	$$\emph{\text{Fr\'echet}}\implies Barrelled\implies Quasibarrelled.$$
	Not all locally $C^*$-algebras are quasibarrelled. The Nachbin-Shirota Theorem (\cite{Shirota}, \cite{Nachbin}) states that the locally $C^*$-algebra $C(X)_\mathcal B$ is Barrelled iff every closed and $C(X)$-bounded subset is compact. Since $C(X)_\mathcal B$ is sequentially complete, it is barrelled if and only if it is quasibarrelled \cite[27.2]{Kothe}.\\
	
	In the section 4, we shall be concerned with Fr\'echet locally $C^*$-algebras, which are locally $C^*$-algebras with underlying topological vector spaces as Fr\'echet spaces. This means that the associated family of seminorms which generates its topology is of  countable cardinality. Theorem \ref{lcd} implies that every Fr\'echet locally $C^*$-algebra $\mathcal A$ is the inverse limit of countable family of $C^*$-algebras $\{\mathcal A_i\}_{i\in \mathbb N}$, i.e., $\mathcal A=\varinjlim \mathcal A_i$.
	
	\subsection{Arens product on bidual of topological algebras.}\label{apd}
	
	Let $\mathcal A$ be a topological algebra with hypocontinuous multiplication (see \cite{Filali},\cite{Gulick} for definition). We will always equip the continuous dual $\mathcal A^*$ (and $\mathcal A^{**}$) with the strong dual topology $\tau_b$ (topology of uniform convergence on bounded sets). Notice that 
	$\tau_b$ is a locally convex topology on $\mathcal A^*$ generated by seminorms $$|f|_B=\sup_{a\in B}|f(a)|$$ where $B$'s are bounded subsets of $\mathcal A$. For any $f\in \mathcal A^*$ and $a\in \mathcal A$ we define functionals ${}_{a}f,f_a:\mathcal A\to \mathbb C$ as $$f_a(b)=f(ab),~{}_{a}f(b)=f(ba)$$ for each $b\in \mathcal A$. From \cite[Lemma 3.1]{Gulick}, the functional ${}_{a}f,f_a\in \mathcal A^*$ and the maps $$a\mapsto f_a~\text{and}~a\mapsto {}_{a}f$$ are continuous map from $\mathcal A$ to $\mathcal A^*$ with respect to the $\tau_b$ topology on $\mathcal A^*$. Similarly, for $f\in \mathcal A^{*}$ and $u\in \mathcal A^{**}$, we define functionals ${}_{u}f,f_u:\mathcal A^*\to \mathbb C$ as $$f_u(b)=u({}_{b}f)~\text{and}~{}_{u}f(b)=u(f_b)$$ for each $b\in \mathcal A$. From \cite[Lemma 3.2]{Gulick}, the functionals ${}_{u}f,f_u\in \mathcal A^*$ for each $u\in \mathcal A^{**}$ and $f\in \mathcal A^*$. Further, the maps $$f\mapsto {}_{u}f~\text{and}~f\mapsto f_u$$ are continuous maps from $\mathcal A^*$ to $\mathcal A^*$ with respect to the strong dual topology $\tau_b$. Finally we define two multiplications $\square$ and $\diamond$ on $\mathcal A^{**}$ as $$u\square v(f)=u({}_{v}f)~\text{and}~u\diamond v(f)=v(f_u)$$ for $u,v\in \mathcal A^{**}$ and $f\in \mathcal A^*$. Due to \cite[Lemma 3.4]{Gulick}, the functionals $u\square v, u\diamond v\in \mathcal A^{**}$ and the maps $u\mapsto u\square v$ and $v\mapsto u\diamond v$ are $\sigma(\mathcal A^{**},\mathcal A^*)$-continuous from $\mathcal A^{**}$ to $\mathcal A^{**}$.	
	\section{ Kaplanski density property in locally $C^*$-algebras}
	Throughout this section $\mathcal A$ will denote a locally $C^*$-algebra. Due to A. Inoue (\cite[Th. 5.1]{Inoue}), $\mathcal A$ is $*$-isomorphic to some $\ast$-subalgebra of $B_{loc}(\mathcal H)$ where $\mathcal H=\coprod H_\alpha$ is a locally Hilbert space and $B_{loc}(\mathcal H)$ denotes the space of locally bounded coherent linear operators (i.e., $T\in B_{loc}(\mathcal H)$ iff $T_{|\alpha}$ is a bounded linear operator for each $\alpha$. See \cite[II.7.5]{Maria2} for details). $B_{loc}(\mathcal H)$ can be considered as a topological subalgebra of $\varprojlim B(\mathcal H_\alpha)$. In fact if $\mathcal A=\underset{\longleftarrow}{\lim}\mathcal A_\alpha$, where $\{\mathcal A_\alpha\}_{\alpha\in \Lambda}$ is a family of $C^*$-algebras, then each $C^*$-algebra $\mathcal A_\alpha$ has an universal representation $\pi_\alpha:\mathcal A_\alpha\to B(\mathcal H_\alpha)$ on some Hilbert space $\mathcal H_\alpha$. Now for each $\alpha\in \Lambda$, define a Hilbert space $\mathcal K_\alpha=\bigoplus_{\gamma\leq \alpha}\mathcal H_\alpha$ and a map $\overline{\pi}_\alpha:\mathcal A\to B(\mathcal K_\alpha)$ as $\overline{\pi}_\alpha(a)=\oplus_{\gamma\leq \alpha} \pi_\gamma(a_\gamma)$. Now consider the locally Hilbert space $\mathcal H=\coprod_{\alpha\in \Lambda}\mathcal K_\alpha$. Then the map $\pi:\mathcal A\to B_{loc}(\mathcal H)$ defined as $\pi(a)(h_\alpha)=\overline{\pi}_\alpha(a)(h_\alpha)$ is an $*$-isomorphism of $\mathcal A$ onto its image. The associated seminorms on $\mathcal A$ are given by $P_\alpha(a)=\|\overline{\pi}_\alpha(a)\|$. This representation $\pi$ will be called as universal representation of $\mathcal A$. (See \cite[II.8.5]{Maria} for complete details.) $\mathcal A$ is a Fr\'echet locally $C^*$-algebra if the index set $\Lambda$ is countable.
	One can easily spot the identification of each $\mathcal A_\alpha$ with $\mathcal A_\alpha=\overline{\overline{\pi}_\alpha(\mathcal A)}$. The explicit description of the $C^*$-algebra $\mathcal A_\alpha$ can be given as the closure of quotient space $\frac{\varprojlim \mathcal A_\alpha}{\sim}$ with respect to the norm $p([(a)])=||a_\alpha||_\alpha$, where $\sim$ denotes the equivalence relation $(a)\sim (b)$ iff $a_\alpha=b_\alpha$.\\
	
	Just like $W^*$-algebras, the locally $C^*$-algebra $B_{loc}(\mathcal H)$ also comes equipped with a bunch of important topologies. We will be mostly concerned with \emph{weak operator topology}(WOT) and \emph{strong operator topology}(SOT) on $B_{loc}(\mathcal H)$ and will adhere to their definitions in \cite[3, Def. 3.1]{Joita}. The same notion has been used in \cite{Maria}.\\
	
	We shall now establish some density results for locally $C^*$-algebras. The techniques are similar to the $C^*$-algebra case with slight modifications adapted in definitions and proofs.
	
	\begin{lemma}
		Let $\mathfrak B$ be a family of compact subsets of $\mathbb R$ such that $\mathfrak B$ contains a neighborhood of each $x\in \mathbb R$ and every finite union of members of $\mathfrak B$ is contained in an another member set of $\mathfrak B$. Denote by $C_0(\mathbb R)_{\mathfrak B}$ the {\emph{lmc}} $*$-algebra of all continuous functions vanishing at infinity, equipped with the topology of uniform convergence on sets in $\mathfrak B$, denoted by $\tau_\mathfrak B$. If $\mathcal S$ is a point-separating subset of $C_0(\mathbb R)$ then the $*$-subalgebra generated by $\mathcal S$ is dense in $C_0(\mathbb R)_\mathfrak B$.
	\end{lemma}
	\begin{proof}
		Let $f\in C_0(\mathbb R)$ be any arbitrarily fixed function and $V$ be any arbitrary open (basic) neighbourhood of $f$. Then $$V=\bigcap_{i=1}^n\bigl\{g\in C_0(\mathbb R)~:~\|g-f\|_{B_i}<\epsilon\bigr\}$$
		for some $\epsilon>0$ and $B_1,..,B_n\in \mathfrak B$ (here $\|f\|_{B_i}=\sup_{B_i}|f(x)|$ for each $i$.). Fix a $B\in \mathfrak B$ such that $\bigcup_{i=1}^n B_i\subset B$. Let $\mathcal S_{B}$ denotes the collection of functions in $\mathcal S$ restricted to set $B$.  Since, $\mathcal S$ separates points in $\mathbb R$, the family of functions $\mathcal S_B$ separates points in $B$. Hence, using the classical Stone-Weierstrass Theorem, we can find a continuous function $g:B\to \mathbb C$ such that $\|g-f\|_B<\epsilon$ and $g$ is a polynomial of functions in $\mathcal S_B\cup \mathcal S_B^*$, i.e., $$g=P(g_1,..,g_r)$$ for some $g_1,..,g_r\in \mathcal S_B\cup \mathcal S_B^*$. Clearly, each $g_i={g_i}'_{|B}$ for some $g_i'\in \mathcal S\cup S^*$. In fact $$g=P(g_1,..,g_r)=P(g_1',..,g_r')_{|B}.$$
		Then, clearly $g'=P(g_1',..,g_r')\in V$. Hence, every open neighbourhood of an arbitrary element $f\in C_0(\mathbb R)_{\mathfrak B}$ contains an element of $*$-subalgebra generated by $\mathcal S$. This concludes that the $*$-subalgebra generated by $\mathcal S$ is dense in $C_0(\mathbb R)_{\mathfrak B}$
	\end{proof}
	\begin{lemma}\label{invocont}
		Let $\mathcal H$ be a locally Hilbert space. The involution $T\mapsto T^*$ is SOT continuous when restricted to set of normal operators on $B_{loc}(\mathcal H)$.
	\end{lemma}
	\begin{proof}
	Follows directly from \cite[Thm. 4.3.1]{Murphy}
	\end{proof}
	\begin{lemma}\label{bddsot}
		If $B$ is a (von Neumann) bounded subset of $B_{loc}(\mathcal H)$, then the map $$B\times B_{loc}(\mathcal H)\to B_{loc}(\mathcal H),~~(S,T)\mapsto S\cdot T$$ is SOT continuous. 
	\end{lemma}
	\begin{proof}
		
		Let $\mathcal H=\coprod \mathcal K_\alpha$ and $h\in \mathcal K_\beta$. There exists a $M_\beta>0$ such that $\|S_{|\mathcal K_\beta}\|_{B(\mathcal K_\beta)}\leq M_\beta$ for each $S\in B$. Hence, if $B\owns S^{(i)}\to S$ in SOT and $B_{loc}(\mathcal H)\owns T^{(i)}\to T$ in SOT, then it follows directly from \cite[Thm. 4.3.1]{Murphy} that 
		$$\|S^{(i)}\cdot T^{(i)}(h_\beta)\|_{\mathcal K_\beta}\to \|S\cdot T(h_\beta)\|_{\mathcal K_\beta}.$$
	\end{proof}
\noindent	We shall need the following functional calculus for locally $C^*$-algebras.
	\begin{theorem}
		Let $\mathcal A=\varprojlim_{i\in \tau}\mathcal A_i$ be a locally $C^*$-algebra, $x$ a normal element in $\mathcal A$ and $S = \mathrm{sp}_A(x)$. Let $C_0(S)$ be the algebra of all continuous functions on $S$ that vanish at $0$, endowed with the topology of uniform convergence on the compacts $S_i=\operatorname{sp}(x_i)$, $i \in \tau$. Then, the \emph{lmc} $*$-algebra $C_0(S)$ is embedded in the locally $C^*$-subalgebra $B$ of $\mathcal A$ generated by $x$, with a topological injective $*$-morphism $\Phi$ such that $\Phi(\mathrm{id}_S) = x$. \cite[10.2]{Maria2}.
	\end{theorem}
	
	\begin{definition}
		A continuous function $f:\mathbb R\to \mathbb C$ is said to be \emph{locally strongly continuous} if for every locally Hilbert space $\mathcal H$ and any (von Neumann) bounded net of self-adjoint operators $\{T^{(i)}\}_{i\in \mathcal I}$ in $B_{loc}(\mathcal H)$ converging in SOT to a self-adjoint operator $T$, the net $\{f(T^{(i)})\}_{i\in \mathcal I}$ converges in SOT to $f(T)$. (See \cite{Inoue} and \cite[10.2]{Maria2} for details on continuous functional calculus of normal operator on locally $C^*$-algebras.)
	\end{definition}
	\begin{lemma}
		If $f:\mathbb R\to \mathbb C$ is a continuous bounded function then it is locally strongly continuous.
	\end{lemma}
	\begin{proof}
		Let $\mathfrak I$ denotes the set of all locally strongly continuous functions. Then $\mathfrak I$ is a vector space with respect to pointwise addition and scalar multiplication. We claim that $C_0(\mathbb R)\subset \mathfrak J$. To show this, let $\mathfrak I_0=\mathfrak I\cap C_0(\mathbb R)$. Notice that $\mathfrak I_0$ is a closed $*$-subalgebra of $C_0(\mathbb R)_\mathcal B$ ($\mathcal B$ denotes the family of all bounded subsets of $\mathbb R$), because $\mathfrak I_0$ is closed with respect to multiplication of functions due to Lemma \ref{bddsot} and it is evidently self-adjoint. To show that $\mathfrak I_0$ is topologically closed in $C_0(\mathbb R)_{\mathcal B}$, let $\{f^{(i)}\}_{i\in \mathcal I}$ be a net of functions in $\mathfrak I_0$ converging to $f\in C_0(\mathbb R)$ in the $\tau_\mathcal B$ topology. For any locally Hilbert space $\mathcal H=\coprod \mathcal K_\alpha$, if $\{T^{(j)}\}_{j\in \mathcal J}$ is a bounded net of self adjoint operators converging to $T\in B_{loc}(\mathcal H)$ in SOT, then for any $h\in \mathcal K_\alpha$, we have \begin{align}\label{inequality}
\|f(T^{(j)})h-f(T)h\|_{\mathcal K_\alpha}&\leq \|f(T^{(j)})h-f^{(i_o)}(T^{(j)})h\|_{\mathcal K_\alpha}+\|f^{(i_o)}(T^{(j)})h-f^{(i_o)}(T)h\|_{\mathcal K_\alpha}+\|f^{(i_o)}(T)h-f(T)h\|_{\mathcal K_\alpha}
		\end{align}
		Since, $\{T^{(j)}\}_{j\in \mathcal J}$ is a bounded sequence, there exists a $M_\alpha>0$ such that $\|T^{(j)}\|_{B(\mathcal K_\alpha)}\leq M_\alpha$ for all $j\in \mathcal J$. Hence, $\operatorname{sp}(T^{(j)}_{|\mathcal K_\alpha})\subset [-M_\alpha,M_\alpha]$ for each $j$. Since, $f^{(i)}\overset{\tau_\mathcal B}{\longrightarrow}f$, we can choose $i_o$ large enough such that $\|f^{(i)}-f\|_{[-M_\alpha,M_\alpha]} <\frac{\epsilon}{\|h\|}$ for all $i\geq i_o$. Hence, using equation \ref{inequality} and the fact that $f^{(i_o)}$ is locally strongly continuous, we can choose a $j_o$ such that 
		\begin{align*}
		\|f(T^{(j)})h-f(T)h\|_{\mathcal K_\alpha}&\leq \|f-f^{(i_o)}\|_{[-M_\alpha,M_\alpha]}\|h\|+\frac{\epsilon}{3}+\|f-f^{(i_o)}\|_{[-M_\alpha,M_\alpha]}\|h\|\\&\leq\epsilon
		\end{align*} for all $j>j_o$. This concludes that $f$ is locally strongly continuous and $\mathfrak I_0$ is a closed subalgebra of $C_0(\mathbb R)_{\mathcal B}$. Now, consider the functions $f(z)=\frac{1}{1+z^2}$ and $g(z)=\frac{z}{1+z^2}$. Clearly, $f,g\in C_0(\mathbb R)$.  We show that $f,g\in \mathfrak I_0$. For any locally Hilbert space $\mathcal H$ and self-adjoint operators $S,T\in B_{loc}(\mathcal H)$, if $h\in \mathcal K_\alpha$ we have
		\begin{align*}
			\|g(T)h-g(S)h\|_{\mathcal K_\alpha}&\leq \|(1+T^2)^{-1}(S-T)(1+S^2)^{-1}h\|_{\mathcal K_\alpha}+\|(1+T^2)^{-1}(T(S-T)S)(1+S^2)^{-1}h\|_{\mathcal K_\alpha}\\&\leq \|(T-S)(1+S^2)^{-1}h\|_{\mathcal K_\alpha}+\|(S-T)S(1+S^2)^{-1}h\|_{\mathcal K_\alpha},
		\end{align*}
		since, $\|(1+T^2)^{-1}\|_{B(\mathcal K_\alpha)}\leq 1$ and $\|(1+T^2)^{-1}T\|_{B(\mathcal K_\alpha)}\leq 1$. Hence $g$ is locally strongly continuous and therefore $g\in \mathfrak I_0$. Since, $g\in \mathfrak I$ and $zg\in \mathfrak I$ due to Lemma \ref{bddsot}. Hence, $f=1-gz \in \mathfrak I$ and therefore both $f,g\in \mathfrak I_0$. Since the set $\{f,g\}$ separates points in $C_0(\mathbb R)$, the $*$-subalgebra generated by $\{f,g\}$ is dense in $C_0(\mathbb R)_\mathcal{ B}$, i.e., $\mathfrak I_0$ is dense in $C_0(\mathbb R)_\mathcal B$. But $\mathfrak I_o$ is closed subalgebra of $C_0(\mathbb R)_{\mathcal B}$, hence $\mathfrak I_0=C_0(\mathbb R)$. 		
		Now let $h:\mathbb R\to \mathbb R$ be a bounded functions, then $hf, hg\in \mathfrak I_0$ and hence $zhg\in \mathfrak I$ (due to Lemma \ref{bddsot}). Thus, $h=hf+zhg\in \mathfrak I$, which concludes that an arbitrary bounded function is locally strongly continuous.
	\end{proof}
	\begin{theorem}\label{totbound}
		Let $\mathcal H=\coprod_{\alpha\in \Lambda}\mathcal K_\alpha$ be a locally Hilbert space and $\mathcal A$ be a locally $C^*$-subalgebra of $\mathcal B_{loc}(\mathcal H)$ such that $\mathfrak A=\overline{\mathcal A}^{SOT}$. Then the following statements hold true.
		\begin{enumerate}
			\item $\mathcal A_{sa}$ is SOT-dense in $\mathfrak A_{sa}$.
			\item For any real number $M>0$   $$\overline{\bigl\{(a_\alpha)\in \mathcal A_{sa}:~\|a_\alpha\|\leq M, \forall \alpha\in \Lambda\bigr\}}^{SOT}=\bigl\{T\in \mathfrak  A_{sa}:~\|T_{|\mathcal K_\alpha}\|\leq M, \forall \alpha\in \Lambda\bigr\}.$$
		\end{enumerate}
	\end{theorem}
	\begin{proof}
		{\bf (1.)} Let $T\in \mathfrak A$ be a self-adjoint operator. Then there exists a net $\{T^{(i)}\}_{i\in \mathcal I}$ of elements in $\mathfrak A$ converging to $T$ in the SOT. Thus ${T^{(i)}}^* \longrightarrow T^*$ in the SOT (due to lemma \ref{invocont}). Hence, ${T^{(i)}}^*\longrightarrow T^*$ in WOT. Therefore $\frac{{T^{(i)}}^*+T^{(i)}}{2}\longrightarrow T$ in WOT. But WOT and SOT agree on convex subsets of $B_{loc}(\mathcal H)$. Hence, $\mathcal A_{sa}$ is SOT-dense in $\mathfrak A_{sa}$.\\
		
	\noindent	{\bf (2.)} Let $T\in \bigl\{T\in \mathfrak  A:~\|T_{|\mathcal K_\alpha}\|\leq M, \forall \alpha\in \Lambda\bigr\}$. Then due to (1), there exists a a net $\{T^{(i)}\}_{i\in \mathcal I}$ of self-adjoint operators in $\mathcal A_{sa}$ converging to $T$ in SOT. Consider a function $f:\mathbb R\to \mathbb R$ such that \begin{equation}
			f(t)=\begin{cases}
				t & t\in [-M,M]\\ \frac{M^2}{t} & t\in(-\infty, -M)\cup (M,\infty),
			\end{cases}
		\end{equation}
		Clearly, $f\in C_0(\mathbb R)$ and hence $f$ is locally strongly continuous. This means that the net $\{f(T^{(i)})\}_{i\in \mathcal I}$ converges to $f(T)$ in SOT. But $f(T)=T$ because $\operatorname{sp}(T)\subset [-M,M]$. Further $\|f(T^{(i)})\|_{B(\mathcal K_\alpha)}\leq \|f\|_\infty=M$ for each $\alpha\in \Lambda$. Thus each $f(T^{(i)})\in \bigl\{(a_\alpha)\in \mathcal A_{sa}:~\|a_\alpha\|\leq M, \forall \alpha\in \Lambda\bigr\}$ and hence $$\overline{\bigl\{(a_\alpha)\in \mathcal A_{sa}:~\|a_\alpha\|\leq M, \forall \alpha\in \Lambda\bigr\}}^{SOT}=\bigl\{T\in \mathfrak  A_{sa}:~\|T_{|\mathcal K_\alpha}\|\leq M, \forall \alpha\in \Lambda\bigr\}.$$
	\end{proof}
\noindent	We now define a more profound type of density property for a locally $C^*$-algebra $\mathcal A\subset B_{loc}(\mathcal H)$, which is a direct generalization of Kaplanski density Theorem for $C^*$-algebras. 
	\begin{definition}
		A locally $C^*$-algebra $\mathcal A\subseteq B_{loc}(\mathcal H)$ for some locally Hilbert space $\mathcal H=\coprod_{\alpha\in \Lambda} \mathcal K_\alpha$ is said to have \emph{Kaplanski Density Property} (KDP) if for any net of non-decreasing positive integers $\{M_\alpha\}_{\alpha\in \Lambda}$ the following holds true, $$\overline{\bigl\{(a_\alpha)_{\alpha\in \Lambda}\in \mathcal A_{sa}:~\|a_\alpha\|_{B(\mathcal K_\alpha)}\leq M_\alpha, \forall \alpha\in \Lambda\bigr\}}^{SOT}=\bigl\{T\in \mathfrak A_{sa}:~\|T_{|\mathcal K_\alpha}\|_{B(\mathcal K_\alpha)}\leq M_\alpha, \forall \alpha\in \Lambda\bigr\}$$ where $\mathfrak A=\overline{\mathcal A}^{SOT}$ and $\mathcal A_{sa}$ denotes the subspace of self-adjoint elements in $\mathcal A$.
	\end{definition}
	\begin{example}
Let $\{\mathcal A_r\}_{r=1}^\infty$ be a sequence of $C^*$-algebras. Consider the system of $C^*$-algebras $\{\mathfrak C_i\}_{i=1}^\infty$, where $\mathfrak C_i=\mathcal A_1\oplus\cdots\oplus \mathcal A_i$. Clearly $\{\mathfrak C_i\}_{i=1}^\infty$ is a reduced inverse system with the natural projections $\theta_{ij}:\mathfrak C_j\to \mathfrak C_i$ defined as $(a_1,\cdots,a_j)\longmapsto (a_1,\cdots, a_i)$ for each $i\leq j$. Let $\mathfrak C$ denote the locally $C^*$-algebra $\varprojlim \mathfrak C_i$. We claim that $\mathfrak C$ has KDP. To see this, first notice that if $\pi_i:\mathcal A_i\to B(\mathcal H_i)$ is the universal representation for each $i$, then $\mathcal H=\coprod \mathcal K_i$ is a locally Hilbert space, where $\mathcal K_i=\oplus_{k\leq i}\mathcal H_i$. The map $\pi:\mathfrak C\to B_{loc}(\mathcal H)$, defined as $$\pi\bigl((T_j)_{j\in\mathbb N}\bigr)(h_1,\cdots, h_i)=(T_{i1}h_1,\cdots, T_{ii}h_i)$$ is the universal representation of $\mathfrak C$, where each $T_j=(T_{j1},...,T_{jj})\in \mathcal A_1\oplus\cdots\oplus \mathcal A_j$. Let $\{M_i\}_{i=1}^\infty$ be a increasing sequence of positive real numbers and suppose $T$ be a self adjoint operator in $B_{loc}(\mathcal H)$ such that $T\in \overline{\mathfrak C}^{SOT}$ and $\|T_{|\mathcal K_j}\|\leq M_j$ for each $j\in \mathbb N$. Define a operator $S\in B_{loc}(\mathcal H)$ such that if $(h_1,...,h_j)\in \mathcal K_j$ then  $S(h_1,...,h_j)=\bigl(\frac{1}{M_1}T_{j1}h_1,\frac{1}{M_2}T_{j2}h_2,...,\frac{1}{M_j}T_{jj}h_j\bigr)$. Clearly $S$ is well defined, coherent and $S\in \overline{\mathfrak C}^{SOT}$. Further $\|S_{|\mathcal K_j}\|_{B(\mathcal K_j)}\leq 1$ holds true for each $j\in \mathbb N$. Hence, by the theorem \ref{totbound}, we can find a net $\{S^{(\alpha)}\}_{\alpha\in \Gamma}$ of operators in $\mathfrak C_{sa}$ which converges to $S$ in the SOT and $\|S^{(\alpha)}_{|\mathcal K_j}\|_{B(\mathcal K_j)}\leq 1$ for each $j\in \mathbb N$ and each $\alpha\in \Gamma$. For each $\alpha\in \Gamma$, let us define a new operator $T^{(\alpha)}\in B_{loc}(\mathcal H)$ such that $T^{(\alpha)}(h_1,...,h_j)=(M_1S^{(\alpha)}_{j1}h_1, ..., M_jS^{(\alpha)}_{jj}h_j)$. Clearly $\|T^{(\alpha)}_{|\mathcal K_j}\|_{B(\mathcal K_j)}\leq M_j$ for each $j\in \mathbb N$. It can easily be seen that $T^{(\alpha)}\in \mathfrak C_{sa}$ for each $\alpha\in \Gamma$ and the net $\{T^{(\alpha)}\}_{\alpha\in \Gamma}\overset{SOT}{\longrightarrow} T$. Thus, $\mathfrak C$ has Kaplanski density property (KDP).
	\end{example}
	Locally $C^*$-algebras with KDP are rather rare. In the next section, we shall see that Kaplansky density property in a Fr\'echet locally $C^*$-algebras enables a full fledge Sherman-Takeda type  characterization for its strong bidual.
	\section{Bidual of locally $C^*$-algebras}\label{mainsection}
	Let $\mathcal A=\varprojlim \mathcal A_\alpha$ be a locally $C^*$-algebra, where $\mathcal A_\alpha$'s make an inverse system of $C^*$-algebras for $\alpha\in \Lambda$. Clearly, $\mathcal A$ has a jointly continuous and hence hypocontinuous multiplication. Hence, $\mathcal A$ satisfies both Gulick's \cite{Gulick} and Filali's \cite{Filali} criterion for existence of Arens products on the second strong dual $\mathcal A^{**}$. Using the involution on each $\mathcal A_\alpha$, we can define an involution on $\mathcal A$ as $u^*(e)=\overline{u(e^*)}$ for each $u\in \mathcal A^{**}$ and $e\in \mathcal A^*$. Using the Arens regularity of  $C^*$-algebras, we prove in the following that a locally $C^*$-algebra is Arens regular.
	\begin{lemma}
For each $\alpha\in \Lambda$, the second conjugate map $f_\alpha^{**}:\mathcal A^{**}\to \mathcal A_\alpha^{**}$ of the canonical projection $f_\alpha:\mathcal A\to \mathcal A_\alpha$ is a continuous $*$-homomorphism (of algebras) with respect to either the Arens products on $\mathcal A^{**}$ and the unique Arens product $\mathcal A^{**}_\alpha$. Further, $f_\alpha^{**}$ preserves involution can be proved easily. Hence, $f_\alpha^{**}$ is a continuous $*$-homomorphism with respect to both the Arens products.
	\end{lemma}
	\begin{proof}
		Continuity of $f_\alpha^{**}$ follows from the fact that conjugate map of of continuous map is again continuous (w.r.t to the strong dual topology). We first show that $f_\alpha^{**}$ is a homomorphism of algebras with respect to first Arens product $\square$ on $\mathcal A^{**}$. Let $u,v\in \mathcal A^{**}$. Then for any $a'\in \mathcal A^*_\alpha$, we have\begin{align}\label{identity}
			f_\alpha^{**}(u\square v)(a')&=u\square v(f_\alpha^*(a'))\nonumber\\&=u\Bigl({}_v\bigl(f_\alpha^*(a')\bigr)\Bigr).
		\end{align}
		But \begin{equation}\left<{}_v\left(f_\alpha^*(a')\right),(b'_\beta)_{\beta\in \Lambda}\right>=v\Bigl(\left(f_\alpha^*(a')\right)_{(b_\beta')}\Bigr)\end{equation} for any $(b_\beta')_{\beta\in \Lambda}\in \mathcal A^*=\varinjlim\mathcal A_\alpha^*$. Through further computations, one can verify that $\left(f_\alpha^*(a')\right)_{(b_\beta')}=f_\alpha^*\Bigl(a'_{f_\alpha((b'_\beta))}\Bigr)$. Substituting this in the previous equation, we get \begin{equation}
{}_v\left(f_\alpha^*(a')\right)=f_\alpha^*\Bigl({}_{f_\alpha^{**}(v)}a'\Bigr).
		\end{equation}
		Finally substituting this in equation \ref{identity}, we get \begin{equation}
f_\alpha^{**}(u\square v)(a')=f_\alpha^{**}(u)\square f_\alpha^{**}(v)(a').
		\end{equation}
		Thus, $f_\alpha^{**}$ is an algebra homomorphism with respect to first Arens product. Similarly, one can show that it is also an algebra homomorphism with respect to second Arens products.
	\end{proof}
\noindent	Using the universal property of inverse limits, we now obtain a continuous linear map $\theta:\mathcal A^{**}\to \varprojlim \mathcal A_\alpha^{**}$ which is a $*$-homomorphism with respect to both the Arens products. Further, $\theta$ is the conjugate(adjoint) map of the continuous surjective linear map $\phi:\varinjlim \mathcal A_\alpha^*\to (\mathcal A^*,\tau_b)$ (recall that $\mathcal A^*$ can be identified as a vector space with the inverse limit $\varinjlim\mathcal A_\alpha^*$ and the inverse limit topology is strictly finer than strong dual topology on $\mathcal A^*$). Conjugate of surjective map is always injective and hence the map $\theta:\mathcal A^{**}\to \varprojlim \mathcal A_\alpha^{**}$ is injective. Hence, the two Arens products on $\mathcal A^{**}$ agree. In general, $\theta$ is not surjective. We can also notice that $\varprojlim\mathcal A_\alpha$ is a topological subalgebra of $\varprojlim\mathcal A_{\alpha}^{**}$ and $\theta (\hat a)=a$ (where $\hat a=J(a)$). We have proved the  following observation.
\begin{corollary}
	Every locally $C^*$-algebra is Arens regular.
\end{corollary}
Let $\pi_\alpha:\mathcal A_\alpha\to B(\mathcal K_\alpha)$ be the universal representation of the $C^*$-algebra $\mathcal A_\alpha$ on some Hilbert space $\mathcal K_\alpha$. By Sherman-Takeda Theorem, $\pi_\alpha$ extends to a continuous faithful $*$- representation $\pi_\alpha^{''}:\mathcal A_\alpha^{**}\to B(\mathcal K_\alpha)$, the image of which is $\overline{\pi_\alpha(\mathcal A)}^{WOT}$. And hence, we get a continuous $*$-representation $\pi'':\varprojlim\mathcal A^{**}_\alpha\to B_{loc}(\mathcal H)$, where $\mathcal H=\coprod\mathcal K_\alpha$ is a locally Hilbert space equipped with inductive limit topology. Clearly, the $\pi''(\varprojlim\mathcal A_\alpha^{**})=\overline{\pi''(\mathcal A )}^{WOT}$ (due to \cite[Th. 3.6, Prop. 3.15]{Joita}).\\

 Consider the injective continuous $*$-representation of bidual given by $\varphi=\pi''\circ \theta: \mathcal A^{**}\to B_{loc}(\mathcal H)$. One need to be careful while dealing with $\varphi$. In case of $C^*$-algebras, $\varphi$ is nothing but the classical isometric embedding of bidual of $C^*$-algebra into the space of bounded operators on a Hilbert space and the image of this representation is whole $\overline{\pi''(\mathcal A)}^{WOT}$ (the classic Sherman-Takeda representation in theorem \ref{STT}). But in the locally $C^*$-algebra settings, $\varphi$ is just a faithful continuous $*$-representation of $\mathcal A^{**}$ into $\overline{\pi''(\mathcal A)}^{WOT}$ (i.e., the inverse map $\varphi^{-1}$ from $\varphi(\mathcal A^{**})$ to $\mathcal A^{**}$ may not be continuous and $\pi''(\mathcal A^{**})$ may be a strict subspace of $\overline{\pi(\mathcal A)}^{WOT}$, i.e., $\varphi(\mathcal A^{**})\subset \overline{\pi''(\mathcal A)}^{WOT}$). \\

Recall that for a given bilinear map $m:X\times Y\to Z$ on vector space $X,Y$ and $Z$, we can define its adjoint map $m^*:Z^*\times X\to Y^*$ as $m^*(z^*,x)(y)=z^*(m(x,y))$.
\begin{lemma}\label{weak-wot}
If $\{a^{(i)}\}_{i\in\mathcal I}$ is a (von Neumann) bounded net in $\mathcal A$ such that $J({a^{(i)}})\overset{\sigma(\mathcal A^{**},\mathcal A^*)}{\looongrightarrow}u\in \mathcal A^{**}$, then $\varphi(u)=\underset{WOT}{\lim} \pi(a^{(i)})$.
\end{lemma}

\begin{proof}
	Consider the bilinear map $m:\mathcal A\times \mathcal H \to \mathcal H$, defined as $m(a,h)=\pi''(a)h$. Then one can easily show through a little exercise that $\varphi(u)h=m^{***}(u,h)$ for any $u\in \mathcal A^{**}$ and $h\in \mathcal H$. Let $(k_\alpha)\in \mathcal H^*=\varprojlim \mathcal K_\alpha$ be any arbitrary element and $h_\beta\in \mathcal H_\beta$. Then, 
	\begin{align*}
		\left<\varphi(u)h_\beta,(k_\alpha)\right>&=\left<m^{***}(u,h_\beta),(k_\alpha)\right>\\&=\left<u,m^{**}(h_\beta,(k_\alpha))\right>\\
		&=\lim_i \left<J(a^{(i)}), m^{**}\bigl(h_\beta, (k_\alpha)\bigr)\right>\\&=\lim_i\left<h_\beta,m^*\bigl((k_\alpha), a^{(i)}\bigr)\right>\\
		&=\lim_i\left<(k_\alpha),m(a^{(i)},h_\beta)\right>\\&=\lim_i\left<\pi''(a^{(i)})h_\beta,(k_\alpha)\right>\\&=\lim_i\left<\pi''(a^{(i)})h_\beta, k_\beta\right>
	\end{align*}
	The above equation tells us that $\{\pi''(a^{(i)})\}_{i\in \mathcal I}$ is a Cauchy net in WOT topology of $B_{loc}(\mathcal H)$, and it is immediate consequence of Banach Steinhauss Theorem (applied component wise) that bounded WOT-Cauchy nets of operators in $B_{loc}(\mathcal H)$ converges to an operator in WOT. Hence, $\{\pi''(a^{(i)})\}_{i\in \mathcal I}$ converges to some $T_u\in B_{loc}(\mathcal H)$ in WOT. Thus from above equation we conclude that $\varphi(u)=T_u=\underset{WOT}{\lim}\pi(a^{(i)})\in \overline{\pi''(\mathcal A)}^{WOT} $.
\end{proof}
An operator $T\in B_{loc}(\mathcal H)$ is said to be \emph{strongly bounded} if there exists a $M>0$ such that $\|T_{\mathcal K_\alpha}\|_{B(\mathcal K_\alpha)}\leq M$ for each $\alpha\in \Lambda$.
\begin{theorem}If $\mathcal A=\varprojlim_{\alpha\in \Lambda} \mathcal A_\alpha$ is a locally $C^*$-algebra then 
$\varphi(\mathcal A^{**})$ contains all the strongly bounded elements of $\overline{\pi''(\mathcal A)}^{WOT}$, i.e., if $T\in \overline{\pi''(\mathcal A)}^{WOT}$ and $T$ is strongly bounded then $T\in \varphi(\mathcal A^{**})$
\end{theorem}
\begin{proof}
	Suppose $T\in \overline{\pi''(\mathcal A)}^{WOT}$ and $T$ is strongly bounded. Then, there exists a $M>0$ such that $\|T_{|\mathcal K_\alpha}\|_{B(\mathcal K_\alpha)}\leq M$ for each $\alpha\in \Lambda$. Since $\overline{\pi''(\mathcal A)}^{WOT}=\overline{\pi''(\mathcal A)}^{SOT}$, using theorem \ref{totbound}, we can find a net of operators $\{\pi''(a^{(i)})\}_{i\in \mathcal I}$ which converges to $T$ in WOT such that $a^{(i)}\in \mathcal A$ for each $i\in \mathcal I$ and $\|\pi''(a^{{i}})_{\mathcal K_\alpha}\|_{B(\mathcal K_\alpha)}\leq M$ for some fixed $M>0$, each $i\in \mathcal I$ and each $\alpha\in \Lambda$. Clearly the $\{a^{(i)}\}_{i\in \mathbb N}$ is (von Neumann) bounded and hence by \cite[23.2.4,p298]{Kothe} the set $\{J(a^{(i)})\}_{i\in \mathcal I}$ is $\sigma(\mathcal A^{**},\mathcal A^{*})$-relatively compact. Hence, without loss of generality we can assume that $\{J(a^{(i)})\}_{i\in \mathcal I}$ converges to some $u\in \mathcal A^{**}$ in the $\sigma(\mathcal A^{**},\mathcal A^{*})$-topology. Now, by lemma \ref{weak-wot}, we conclude that $\varphi(u)=T$.
\end{proof}

We now show that if $\mathcal A$ is a Fr\'echet locally $C^*$-algebra having Kaplanski density property, then $\mathcal A^{**}$ is topologically $*$-isomorphic to the algebra $\overline{\pi''(\mathcal A)}^{WOT}$.

	\begin{theorem}\label{main-new}
		If $\mathcal A=\varprojlim_{r\in \mathbb N} \mathcal A_r$ is a  Fr\'echet locally $C^*$-algebra having KDP, then its strong bidual $\mathcal A^{**}$ is topologically and algebraically $*$-isomorphic to the locally von Neumann algebra $\overline{\pi''(\mathcal A)}^{WOT}=\varprojlim\mathcal A_r^{**}$.
	\end{theorem}
\begin{proof}
	First we will show that the map $\varphi:\mathcal A^{**}\to B_{loc}(\mathcal H)$ is surjective onto $\overline{\pi''(\mathcal A)}^{WOT}$, where $\mathcal H=\coprod_{r\in \mathbb N}\mathcal K_r$ and $\mathcal K_r$ is the associated Hilbert space of the universal representation of $C^*$-algebras $\mathcal A_r$ for each $r\in \mathbb N$. It would suffice to show that each self-adjoint operator in $\overline{\pi''(\mathcal A)}^{WOT}$ has a pre-image in the map $\varphi$. Let $T\in \overline{\pi''(\mathcal A)}^{WOT}$ be a self-adjoint operator. Then, due to KDP, we get a net $\{a^{(i)}\}_{i\in \mathcal I}$ of self adjoint elements in $\mathcal A$ such that $\pi''(a^{(i)})\longmapsto T$ in SOT (and hence in WOT) such that for any $i\in \mathcal I$ and $r\in \mathbb N$, we have $\|\pi''(a^{(i)})_{|\mathcal K_r}\|_{B(\mathcal K_r)}\leq \|T_{\mathcal K_r}\|_{B(\mathcal K_r)}$. Hence, the net $\{a^{(i)}\}_{i\in \mathcal I}$ is (von Neumann) bounded in $\mathcal A$. Therefore, $\{J(a^{(i)})\}_{i\in \mathcal I}$ is $\sigma(\mathcal A^{**},\mathcal A^*)$-relatively compact in $\mathcal A^{**}$. Without loss of generality we can assume that it converges to some $u\in \mathcal A^{**}$ in the $\sigma(\mathcal A^{**},\mathcal A^*)$ topology. Thus, by theorem \ref{weak-wot}, $\varphi(u)=T$. Since $\overline{\pi''(\mathcal A)}^{WOT}= \varprojlim\mathcal A^{**}_r$ (\cite[Prop. 3.14]{Joita}), it is a Fr\'echet space. Further it is known that, strong dual of a Fr\'echet space is a DF-space and strong dual of a DF-space is a Fr\'echet space and hence $\mathcal A^{**}$ is a Fr\'echet space. Thus, $\varphi:\mathcal A^{**}\to \overline{\pi''(\mathcal A)}^{WOT}$ is a continuous bijection between Fr\'echet spaces and hence an open map (due to open mapping theorem). This establishes the topological embedding of $\mathcal A^{**}$ via  $\varphi:\mathcal A^{**}\to B_{loc}(\mathcal H)$.
	\end{proof}
	\section{Continuity of some bilinear maps associated to a representation of locally $C^*$-algebras.}	
Let $\mathcal A$ be a locally $C^*$-algebra. Notice that the universal representation $\pi:\mathcal A\to B_{loc}(\mathcal H)$ gives rise to a separately continuous bilinear map $\theta:\mathcal A\times\mathcal H\to\mathcal H$, defined as $\theta(a,h)=\pi(a)(h)$ for $a\in \mathcal A$ and $h\in \mathcal H$. Let $b(\mathcal A)$ denote the bounded part of $\mathcal A$ i.e., $(a_\alpha)\in b(\mathcal A)$ iff there exists a $M>0$ such that $\|a_\alpha\|\leq M$ for each $\alpha\in \Lambda$. Then $b(\mathcal A)$ is a $*$-subalgebra of $\mathcal A$ and the subspace topology is generated by a norm $\|(a_\alpha)\|=\sup_\alpha\|a_\alpha\|$.\\

A bilinear map $m:\mathcal E\times \mathcal F\to \mathcal G$ on topological vector spaces $\mathcal E,\mathcal F, \mathcal G$ is said to be hypocontinuous if for any bounded subset $B$ of $\mathcal E$ (or $\mathcal F$) and a $0$-neighbourhood $W$ of $\mathcal G$, there exists a $0$-neighborhood $V$ of $\mathcal E$ (or $\mathcal F$) such that $m(B,V)\subset W$ (or $m(V,B)\subset W$). We have some interesting observations regarding $\theta$ and their adjoints. 
	\begin{lemma}\label{bilnet}
		For a Fr\'echet locally $C^*$-algebra $\mathcal A$, the following statements hold true.
		\begin{enumerate}
			\item  The bilinear map $\theta$ is hypocontinuous.	The bilinear map $\theta_{|b(\mathcal A)\times \mathcal H}:b(\mathcal A)\times \mathcal H\to \mathcal H$ is jointly continuous.
			\item The adjoint map $\theta^*:\mathcal H^*\times \mathcal A\to \mathcal H^*$ is jointly continuous.
			\item The third adjoint map	$\theta^{***}:\mathcal A^{**}\times \mathcal H\to\mathcal H$ is jointly continuous iff $\theta$ is jointly continuous. 
			
		\end{enumerate}
	\end{lemma}
	
	\begin{proof}
		{(\bf 1.)} Since both $\mathcal A$ and $\mathcal H$ are barrelled, hypocontinuity of $\theta$ follows from $\cite[5.5.2]{Schaefer}$.	Let $W$ be a $0$-neighbourhood in $\mathcal H$. Fix the unit ball $B=\{a\in b(\mathcal A)~:~\|a_\alpha\|\leq 1\forall \alpha\}$, which is itself a $0$-neighbourhood. Then, the family of operators $\mathcal F=\{\pi(a):~a\in B\}$ is pointwise bounded because for any fixed $h\in \mathcal H$, the map $\mathcal A\to \mathcal H$ defined as $a\mapsto \pi(a)h$ is continuous and it maps bounded sets to bounded sets (a continuous map with metrizable domain is always bounded). Since, $\mathcal H$ is barrelled, we can invoke Banach-Steinhaus Theorem and deduce that the family $\mathcal F$ is equi-continuous at $0$. Hence, there exists a $0$-neighbourhood $V$ in $\mathcal H$ such that $\mathcal F(V)\subset W$. Hence, $\theta(B\times V)\subset W$, which implies that $\theta_{|b(\mathcal A)\times \mathcal H}$ is jointly continuous.\\
		
		\noindent{(\bf 2.)} Recall that the inductive limit of a sequence of DF-spaces is a DF-space (see \cite[p196-197]{Schaefer}) and hence $\mathcal H$ is a DF-space. The strong dual of a DF-space is a Fr\'echet space (see \cite[12.4.1]{Jarchow}), i.e., $\mathcal H^*$ is a Fr\'echet space. The continuity of adjoint maps $\theta^*$ follows from the fact that a separately continuous bilinear map from the product of two metrizable space, one of which is baire, to a locally convex space is jointly continuous (\cite[Ch. III, 5.1]{Schaefer}).\\
		
		\noindent{(\bf 3.)}  From \cite[23.4.1-2]{Kothe}, we know that for a barrelled spaces $X$ the neighbourhood base at any point in the second (strong) dual $X^{**}$ is given by the family of polars $\{U^{\circ\circ}\}$, where $\{U\}$ is an absolutely convex neighbourhood base for $X$. Hence, for any $0$-neighbourhood $W^{\circ\circ}\cap \mathcal H$ in $\mathcal H$, using the continuity of $\theta$, we can find neighbourhoods $U$ and $V$ of $\mathcal A$ and $\mathcal H$ respectively such that $\theta(U\times V)\subset W^{\circ\circ}\cap \mathcal H$. Then one can easily notice that $\theta^{***}\bigl(U^{\circ\circ}\times V\bigr)\subset W^{\circ\circ}\cap\mathcal H$. To see this, suppose $u^{\circ\circ}\in U^{\circ\circ}$ and $v\in V$. Then for any $w^\circ\in W^\circ$, we have 
		\begin{equation}\label{polar}
			\theta^{***}(u^{\circ\circ},v)(w^\circ)=u^{\circ\circ}(\theta^{**}(v,w^\circ))
		\end{equation}
		But for any $u\in U$, we have \begin{equation}
			\left|\theta^{**}(v,w^\circ)(u^\circ)\right|=\left|v\left(\theta^{*}(w^\circ,u^\circ)\right)\right|=|w^\circ(\theta(u,v))|\leq 1 
		\end{equation}
		Hence, $\theta^{**}(v,w^\circ)\in U^\circ$. Thus, using the equation \ref{polar}, we conclude that $$\left|\theta^{***}(u^{\circ\circ},v)(w^\circ)\right|\leq 1.$$
		
		This implies that $\theta^{***}(u^{\circ\circ},v)\in W^{\circ\circ}\cap\mathcal H$, i.e., $\theta^{***}(U^{\circ\circ}\times V)\subset W^{\circ\circ}\cap\mathcal H$ and the map $\theta^{***}$ is jointly continuous. The converse holds true because due to the barrelledness of $\mathcal A$ the map $\theta$ is nothing but the restriction of $\theta^{**}$ to $\mathcal A\times \mathcal H$.
	\end{proof}
	Note that, in the above results we still do not know if the second adjoint bilinear map $\theta^{**}$ is jointly continuous or not, even when $\mathcal A$ is a Fr\'echet locally $C^*$-algebra.\\ 
	
	We note the following elementary observation, which in the case of Banach spaces is trivial to prove using norm inequalities, but cannot be taken for granted in the case of topological vector spaces.
	\begin{lemma}\label{biltolin}
		Let $\mathcal U$ and $\mathcal V$ be locally convex topological vector spaces and $m:\mathcal U\times \mathcal V\to \mathcal V$ be a jointly continuous bilinear map. Then the induced map $\hat m:\mathcal U\to B(\mathcal V)$ is continuous, where $B(\mathcal V)$ is the space of continuous linear operators on $\mathcal V$ and it is equipped with topology of uniform convergence on bounded sets of $\mathcal V$.
	\end{lemma}
	
	\begin{proof}
		Since $\widehat{m}$ is linear (by bilinearity of $m$), it suffices to prove continuity at $0\in\mathcal{U}$. 
		We need to show that if $(u_\lambda)$ is a net in $\mathcal{U}$ with $u_\lambda \to 0$, then 
		$\widehat{m}(u_\lambda)\to 0$ in $B(\mathcal V)$. 			
		A basic neighbourhood of $0$ in $B(\mathcal V)$ is of the form
		\[
		\mathcal{W}(K,W)=\bigl\{T\in B(\mathcal V):T(K)\subseteq W\bigr\},
		\]
		where $K\subseteq \mathcal{V}$ is a bounded set and $W\subseteq \mathcal{V}$ is a convex balanced neighbourhood of $0$. 
		(We may assume without loss of generality that $K$ is absolutely convex.)		
		Fix such a neighbourhood $\mathcal{W}(K,W)$. We need to find $\lambda_0$ such that 
		$\widehat{m}(u_\lambda)\in\mathcal{W}(K,W)$ for all $\lambda\ge\lambda_0$, 
		i.e.,\ $m(u_\lambda,K)\subseteq W$.	Since $m$ is jointly continuous at $(0,0)$ and $m(0,0)=0$, there exist convex balanced 
		neighbourhoods $P\subseteq \mathcal{U}$ of $0$ and $Q\subseteq \mathcal{V}$ of $0$ such that
		\[
		m(P\times Q)\subseteq W.
		\]

		Since $K$ is bounded and $Q$ is a neighbourhood of $0$ in $\mathcal{V}$, there exists $t>0$ such that
		\[
		K\subseteq tQ.
		\]

		Define the scaled neighbourhood $			P':=\frac{1}{t}P=\bigl\{u\in\mathcal{U}:tu\in P\bigr\}.$
		Then $P'$ is a convex balanced neighbourhood of $0$ in $\mathcal{U}$.	Let $u\in P'$ and $v\in K$.  
		Then $tu\in P$ (write $tu=u'$ with $u'\in P$) and $v=tq$ for some $q\in Q$. 
		By bilinearity of $m$,
		\begin{align*}
			m(u,v)
			&=m\!\left(\frac{1}{t}u',tq\right)
			=\frac{1}{t}\,m(u',tq)
			=\frac{1}{t}\cdot t\cdot m(u',q)
			=m(u',q).
		\end{align*}
		Since $u'\in P$ and $q\in Q$, we have $m(u',q)\in W$, hence $m(u,v)\in W$. 
		As $v\in K$ was arbitrary, $m(u,K)\subseteq W$.				
		Since $u_\lambda\to 0$ in $\mathcal{U}$, there exists $\lambda_0$ such that 
		$u_\lambda\in P'$ for all $\lambda\ge\lambda_0$. 
		For these indices, we have
		\[
		\widehat{m}(u_\lambda)(K)=m(u_\lambda,K)\subseteq W,
		\]
		so $\widehat{m}(u_\lambda)\in\mathcal{W}(K,W)$. 
		Thus $\widehat{m}(u_\lambda)\to 0$ in $B(\mathcal V)$.
	\end{proof}
	The above lemma holds significance because it gives us a recipe to construct representations of a topological algebra on $B_{loc}(\mathcal H)$ through coherent bilinear maps. In fact if $\mathcal U$ is a topological algebra, $\mathcal H=\coprod \mathcal H_\alpha$ is a locally Hilbert space and $m:\mathcal V\times \mathcal H\to \mathcal H$ is a jointly continuous bilinear map such that $m(\mathcal V, \mathcal H_\alpha)\subset \mathcal H_\alpha$ for each $\alpha$, then $\rho:\mathcal V\to B_{loc}(\mathcal H)$, defined as $\rho(v)h=m(v,h)$, is a continuous representation.\\

	Recall that in Lemma \ref{bilnet}, using the continuity of representation $\pi:\mathcal A\to B_{loc}(\mathcal H)$, we showed that the associated bilinear map $\theta:\mathcal A\times \mathcal H\to \mathcal H$ is jointly continuous. This was made possible due to the assumption that $\mathcal A$ is Fr\'echet space. In the next result we show a similar phenomena takes place if one of the space is locally bounded, i.e., there exists a bounded neighbourhood of zero. Although  interesting, next result does not help us to evade the Fr\'echet condition on $\mathcal A$ in Lemma \ref{bilnet} because locally Hilbert spaces are not locally bounded (this follows from \cite[4.6, Thm. 2]{Jarchow}, which states that a bounded subset of a strict inductive limit of an increasing sequence of locally convex spaces $\{E_n\}_{n\in \mathbb N}$ is necessarily contained in some $E_m$. Also the fact that Hausdorff locally convex spaces which have bounded convex neighbourhood of zero are necessarily normed, a famous result by Kolmogorov \cite{Kolmogorov}, tells that locally Hilbert spaces do not have bounded convex neighbourhood.)\\
	
	A topological vector space $\mathcal V$ is said to be locally bounded if there exists a (von Neumann) bounded neighbourhood of $0$.
	
	\begin{theorem}
		Let $X$ and $Y$ be locally convex topological vector space such that $Y$ is locally bounded (i.e., $Y$ is seminormed space). If $\phi:X\to B(Y)$ is a continuous linear map then the associated bilinear map $m:X\times Y\to Y$ $$m(x,y)=\phi(x)(y)$$is jointly continuous. ($B(Y)$ is equipped with topology of uniform convergence on bounded sets.)
	\end{theorem}
	\begin{proof}
		It is sufficient to prove that the map $\varphi:B(Y)\times Y\to Y$ given by $\varphi(T,y)=T(y)$ is continuous because $\phi=\varphi\circ \mu$, where $\mu:X\times Y\to B(Y)\times Y$ is the continuous map $\mu(x,y)=(Tx,y)$. We will prove that continuity of $\varphi$ at $(0,0)$, which will imply that $\varphi$ is continuous everywhere.  Let $V$ be a neighbourhood of $0$ in $Y$. Fix a continuous seminorm $q$ on $Y$ such that $\{y\in Y:~q(y)<1\}\subset V$. Since $Y$ is locally bounded and locally convex, we can fix a bounded, convex, balanced and absorbing neighbourhood $E$ of $0$ in $Y$. Define a seminorm $\mu_E(T)=\sup_{z\in E}q(Tz)$ on $B(Y)$. Clearly, $\mu_E:B(Y)\to [0,\infty)$ is continuous (due to the topology of uniform convergence on bounded sets). Let $N=\{T\in B(Y)~:~\mu_E(T)<\frac{1}{2}\}$. Clearly $N$ is a neighbourhood of $0$ in $B(Y)$. Consider the Gauge functional $g_E(y)=\inf\{\lambda>0~:~y\in \lambda E\}$. Then $g_E$ is a continuous seminorm on $Y$ such that $\{y~:~g_E(y)<1\}\subset E$. Let $W=\{y\in Y~:~g_E(y)<\frac{1}{2}\}$ be a neighbourhood of $0$ in $Y$. We claim that $\varphi (N\times W)\subset V$. To see this, take $T\in N$ and $y\in W$. Then, $g_E(y)<\frac{1}{2}$ and hence $y= ke$ for some $0<k<\frac{1}{2}$ and $e\in E$. Thus, $T(y)=kT(e)$ and $$q(T(y))=kq(T(e))\leq k\sup_{z\in E}q(T(z))\leq\frac{1}{2}\mu_E(T)<\frac{1}{4}.$$ Hence, $T(y)\in \{y\in Y:q(T(y))<1\}\subset V$. Thus, $\varphi(N\times W)\subset V$. This concludes that $m:X\times Y\to Y$ is jointly continuous.
	\end{proof}
	Let $\Omega(\mathcal A)$ denote the set of all continuous multiplicative linear functionals on the topological algebra $\mathcal A$.
	\begin{theorem}
		Let $\mathcal A$ be a topological algebra with hypocontinuous multiplication such that $\Omega(\mathcal A)$ separates points of $\mathcal A^{**}$, then any $\sigma(\mathcal A^{**},\mathcal A^*)$ cluster point of idempotents of $\hat{\mathcal A}$ in $\mathcal A^{**}$ is again an idempotent with respect to both the Arens products.
	\end{theorem}
	\begin{proof}
		Let $u\in \mathcal A^{**}$ be a $\sigma(\mathcal A^{**},\mathcal A^*)$-cluster point of idempotents in $\hat{\mathcal A}$. Then there exists a net $\{\hat{a}_i\}$ such that $a_i\in \mathcal A$, $a_i^2=a_i$ for each $i$ and $\hat{a}_i\overset{\sigma(\mathcal A^{**},\mathcal A)}{\looongrightarrow} u$. Hence, 
		\begin{equation}
			f(a_i)\to u(f)~~\forall~f\in \mathcal A^*.
		\end{equation}
		But for $f\in \Omega(\mathcal A)$, we have $$f(a_i)=f(a_i^2)=(f(a_i))^2~~\forall i.$$
		Hence, either $f(a_i)=0$ or $f(a_i)=1$ holds true for each $i$. Since, the net $\{f(a_i)\}$ is convergent, we deduce that it is either eventually 0 or eventually 1. Hence, \begin{equation}\label{zoo}
			u(f)=0 ~\text{or}~ u(f)=1 ~\text{for any}~f\in \Omega(\mathcal A).
		\end{equation}
		Now, notice that
		\begin{equation}\label{prod}
			u\square u(f)=u({}_{u}f) ~\forall f\in \mathcal A^*.
		\end{equation}
		But for $f\in \Omega(\mathcal A)$, \begin{align*}
			{}_{u}f(x)&=u(f_x)\\
			&=u\left(f(x).f\right)\\
			&=f(x)u(f).
		\end{align*}
		Hence, \begin{equation}
			{}_{u}f=u(f).f
		\end{equation}
		Using Equation \ref{prod} we get that for all $f\in \Omega (\mathcal A)$ \begin{align*}
			u\square u(f)&=u(u(f).f)\\
			&=\left(u(f)\right)^2.
		\end{align*}
		But due to Equation \ref{zoo}, both $u\square u(f)$ and $u(f)$ can take values either 0 or 1 for $f\in \Omega(\mathcal A)$ . Hence, \begin{equation}
			u\square u(f)=u(f) ~~\forall f\in \Omega(\mathcal A).
		\end{equation}
		By assumption, $\Omega(\mathcal A)$ separates points of $\mathcal A^{**}$ and hence $u\square u=u$, i.e., $u$ is an idempotent with respect to first Arens product. Similarly, one can show that $u$ is an idempotent with respect to second Arens product.
	\end{proof}
	\begin{example}
		Simplest example of this phenomena can be seen in the case of $C^*$-algebra $\mathcal A=c_0(\mathbb Z)$. We know that $\mathcal A^{**}=\ell^\infty(\mathbb Z)$ and $\Omega(\mathcal A)=\mathbb Z$. Clearly, $\Omega(\mathcal A)$ separates points in $\mathcal A^{**}$. The idempotents in $c_0(\mathbb Z)$ are the functions with finite support taking values 0 and 1. While the idempotents in $\ell^\infty(\mathbb Z)$ are all the functions which takes only values 0 and 1.
	\end{example}
	\noindent We end this article with the following relevant open problems and future pursuits.\\
	
	\noindent	Q1. Give more examples of topological algebras (perhaps non-commutative) with hypocontinuous multiplication such that the character space $\Omega(\mathcal A)$ separate points in the strong bidual.\\
		
	\noindent	Q2. Characterize Quasibarrelled locally $C^*$-algebra and examples of such which are not Fr\'echet (even better if non-commutative).\\
	
	For a $C^*$-algebra $\mathcal A$, the set of representations of $\mathcal A$ and $\mathcal A^{**}$ have one-to-one correspondence, because of the fact that $\mathcal A$ embeds isometrically into $\mathcal A^{**}$. This is not true for locally $C^*$-algebras because if $\mathcal A$ is non-quasibarrelled locally $C^*$-algebras, then a representation of $\mathcal A^{**}$ on $B_{loc}(\mathcal H)$ cannot be compressed to a representation of $\mathcal A$. The topology of $\mathcal A$ is courser than the strong topology it inherits from $J(\mathcal A)$.\\
	
	\noindent Q3. Characterize/classify locally $C^*$-algebras which have Kaplanski Density property (KDP) and give more concrete examples.\\
	~\\~\\
	\noindent	{\bf Acknowledgements:} This research was carried out with the financial support of Slovenian
	Agency for Research and Innovation (grant P1-0222-Algebra, operator theory and financial
	mathematics at Institute of mathematics, physics and mechanics, Ljubljana, Slovenia).
	The second author is also a member of the University of Ljubljana, KTT21 project
	group (grant SN-ZRD/22-27/0510). We would like to thank Prof.(ret.) Ajit Iqbal Singh for bringing to our notice various references related  to locally $C^*$-algebras.

	\Addresses
\end{document}